\documentclass[oneside,reqno]{amsart}
\usepackage[utf8]{inputenc}
\usepackage{amsmath,amssymb,amsthm, graphicx,float,url}
\usepackage[dvipsnames]{xcolor}
\usepackage{tikz-cd} 
\usepackage{thm-restate}
\usepackage{physics} 
\usepackage[hidelinks]{hyperref}

\usepackage[capitalize]{cleveref}
\crefname{subsection}{subsection}{subsections}
\usepackage[symbol]{footmisc}
\usepackage{nicefrac} 
\allowdisplaybreaks 

\usepackage{lineno}

\newcommand{\systemU}{(u_i)_m\xspace}

\newcommand{\ONF}{\ensuremath{\mathsf{NF}}\xspace}
\newcommand{\DOddGroup}{G}

\numberwithin{equation}{section}

\usepackage[a4paper]{geometry}
\usepackage{mathtools}
\usepackage{stackrel}

\newcommand{\bi}{\begin{itemize}}
	\newcommand{\ei}{\end{itemize}}
\newcommand{\be}{\begin{enumerate}}
	\newcommand{\ee}{\end{enumerate}}

\usepackage[english]{babel}
\usepackage{tabularx,gastex,hhline,rotating,enumerate,xspace}

\usepackage{tikz}
\usetikzlibrary{calc,decorations.pathmorphing,patterns,arrows,decorations.markings,positioning}
\usetikzlibrary{automata,chains,fit,shapes}

\renewcommand{\geq}{\geqslant} \renewcommand{\leq}{\leqslant}

\newenvironment{aw}{\noindent\color{blue} AW :  }{}

\newtheorem{theorem}{Theorem}[section]

\newtheorem{corollary}[theorem]{Corollary}
\newtheorem{lemma}[theorem]{Lemma}

\newtheorem{open}{Open question}[section]
\theoremstyle{definition}
\newtheorem{definition}[theorem]{Definition}
\newtheorem{remark}[theorem]{Remark}

\newtheorem{theoremx}{Theorem}

\newtheorem{propx}[theoremx]{Proposition}

\newcommand{\Aut}{\textnormal{Aut}}

\renewcommand{\geq}{\geqslant} \renewcommand{\leq}{\leqslant}  

\newcommand{\Z}{\mathbb Z}
\newcommand{\N}{\mathbb N}

\newcommand{\bbX}{\mathbb{X}}
\newcommand{\bbY}{\mathbb{Y}}
\newcommand{\Fin}{\mathrm{Fin}}
\newcommand{\Arb}{\mathrm{FinPres}}
\newcommand{\Abe}{\mathrm{Abel}}
\newcommand{\Hom}{\mathrm{Hom}}

\newcommand{\Epi}[2]{\operatorname{Epi}(#1,#2)}
\newcommand{\Sol}[1]{\operatorname{Sol}(#1)}

\newcommand{\cD}{\mathcal{D}}

\newcommand{\cR}{\mathcal{R}}

\newcommand{\cC}{\mathcal{C}}

\newcommand{\cG}{\mathcal{G}}

\newcommand{\cT}{\mathcal{T}}

\newcommand{\cN}{\mathcal{N}}

\newcommand{\cX}{\mathcal{X}}
\newcommand{\cY}{\mathcal{Y}}

\newcommand{\GL}{\mathrm{GL}}

\newenvironment{js}{\noindent\color{red} JS :  }{}

\newcommand{\MFd}{\mathfrak{d}}

\newcommand{\MFr}{\mathfrak{r}}

\newcommand{\compproblem}[3][]{%
	\par\vspace{0.125cm plus 0.1cm minus 0.05cm}\begin{tabularx}{\textwidth-2\parindent}{lX}%
		\if\relax\detokenize{#1}\relax%
		\else%
		\textnormal{\textbf{Constant:}}&#1\\%
		\fi%
		\textnormal{\textbf{Input:}}&#2\\%
		\textnormal{\textbf{Question:}}&#3\\%
	\end{tabularx}\vspace{0.125cm plus 0.1cm minus 0.05cm}\par%
}

\newcommand{\compalgo}[3][]{%
	\par\vspace{0.125cm plus 0.1cm minus 0.05cm}\begin{tabularx}{\textwidth-2\parindent}{lX}%
		\if\relax\detokenize{#1}\relax%
		\else%
		\textnormal{\textbf{Constant:}}&#1\\%
		\fi%
		\textnormal{\textbf{Input:}}&#2\\%
		\textnormal{\textbf{Output:}}&#3\\%
	\end{tabularx}\vspace{0.125cm plus 0.1cm minus 0.05cm}\par%
}

\newcommand{\NP}{\ensuremath{\mathsf{NP}}\xspace} %

\renewcommand{\P}{\ensuremath{\mathsf{P}}\xspace}
\newcommand{\PSPACE}{\ensuremath{\mathsf{PSPACE}}\xspace} 
\newcommand{\EXPSPACE}{\ensuremath{\mathsf{EXPSPACE}}\xspace} 

\newcommand{\gen}[1]{\left<\, \mathinner{#1} \,\right>}
\newcommand{\Gen}[2]{\left< \,\mathinner{#1} \mid \mathinner{#2}\,\right>}

\newcommand{\mcomm}[3]{\left[{#1},\kern.1em_{#2}\,\kern.1em {#3} \right]}
\newcommand{\smcomm}[2]{\left[_{#1}\,\kern.1em {#2} \right]}

\begin{document}

\title{Epimorphism testing  with dihedral  targets}

\author{Murray Elder}
\address{School of Mathematical and Physical Sciences, University of Technology Sydney, Ultimo NSW 2007, Australia}  
\email{murray.elder@uts.edu.au}
%
%
%
\author{Jerry Shen}
\address{School of Mathematical and Physical Sciences, University of Technology Sydney, Ultimo NSW 2007, Australia} 
\email{qing.shen-1@uts.edu.au}

\author{Armin Wei\ss}
\address{ Institut f\"ur Formale Methoden der Informatik, Universit\"at Stuttgart, 
70569 Stuttgart, Germany}
\email{armin.weiss@fmi.uni-stuttgart.de}
%

\date{\today}



\begin{abstract}
We consider the problem of deciding whether or not there is an epimorphism from the class of finitely presented groups to a fixed finite group.
We show this problem is $\mathsf{NP}$-complete when the target is a non-nilpotent dihedral group. 
This complements work of Kuperberg and Samperton (2018, Geom. Topol.) who showed the same result when the target is a fixed finite non-abelian simple group. 
We also include a 
survey of complexity and reduction results  when the source and target are classes of finitely presented groups.
\end{abstract}
\maketitle

\section{Introduction}

In this article we consider the complexity of the following decision problem for classes of finitely generated groups $\cD, \cT$.
The (uniform) \emph{epimorphism problem} from $\cD$ to $\cT$, denoted $\Epi{\cD}{\cT}$, is the following decision problem.
\compproblem[]{Finite descriptions 
for groups $G \in \cD$ and $H \in \cT$}{Is there 
an epimorphism 
from $G$ to $H$?}
Note that, in the category of groups, an epimorphism is a surjective homomorphism.
We refer to $G\in \cD$ as the \emph{domain group} and $H\in \cT$ is the \emph{target group} for the problem. 
If $\cT=\{H\}$ is a singleton, we write $\Epi{\cD}{H}$ for the epimorphism problem from a class $\cD$ to the fixed group $H$, in which case the input is just a finite description for a group $G\in\cD$.

When the target is a finite group, 
epimorphism is easily seen to be 
in \NP (see \cref{lem:epi_finite}). 
Kuperberg and Samperton \cite{Kuperberg} considered the case of epimorphism from certain 3-manifold groups to finite non-abelian simple groups, in the context of more general questions. 
It follows from their work that the epimorphism problem from a finitely presented group to a fixed finite non-abelian simple group is \NP-hard (see Subsection~\ref{subsec:Kuperberg}), and therefore \NP-complete.
Here we show the same result applies when the target is a non-nilpotent finite dihedral group.

\begin{theoremx}\label{thm:MainDihedral}
   The epimorphism problem from finitely presented groups to a finite non-nilpotent dihedral group is \NP-hard (and therefore \NP-complete).
\end{theoremx}
Recall that a dihedral group $D_{2n}$ of order $2n$ is nilpotent if and only if $n$ is a power of $2$.  
Let $\Arb$ denote the class of finitely presented groups, and $Z(G)$ the center of a group $G$.
\begin{propx}\label{prop:CenterQuotient}
Let $G\in \Arb$ with 
$Z(G)$ finitely generated. 
      If $\Epi{\Arb}{G/Z(G)}$ is \NP-hard, then $\Epi{\Arb}{G} $ is \NP-hard.
\end{propx}

Using this and observing that $D_{4n}/Z(D_{4n})\cong D_{2n}$, the proof of \cref{thm:MainDihedral} reduces to showing $\Epi{\Arb}{D_{2n}}$ is \NP-hard for odd $n>1$.

Remeslennikov
\cite{Remeslennikov} proved that for $\cN$ the class of non-abelian nilpotent groups, the epimorphism problem from $\cN$ to $\cN$
is undecidable, via Hilbert's 10th problem.
Applying the work of Razborov \cite[Theorem 3]{Razborov1985} on equations in free groups, the epimorphism problem from finitely presented groups to finitely generated free groups is decidable, but without any known complexity bounds (see \cref{sec:otherResults}). 
Friedl and L\"oh \cite{FriedlLoh}
considered the epimorphism problem from finitely presented groups to virtually abelian groups, proving that the problem is decidable when the target is either a virtually cyclic group 
or the direct product of an abelian group and a finite group.

We conclude this paper by briefly surveying additional 
results for the epimorphism problem for groups where at least one of the source or target is a class of infinite finitely generated groups (noting that the problem is also interesting when both source and target are (classes of) finite groups, for example recent work of Grochow, Srivastava and Thakkar \cite{GrochowICALP})
by collecting
together known results 
and  some straightforward consequences of known results.

\begin{propx}\label{prop:ObserveFacts}
    The epimorphism problem from finitely presented groups 
    to 
    \begin{enumerate}
               \item finite rank free groups is decidable (via \cite{Razborov1985}, see \cref{lem:epiFree})

\item a fixed non-abelian finite simple group is \NP-complete \cite{Kuperberg} 
    
       \item virtually cyclic groups, direct products of a free abelian group with a finite group, and semi-direct products of a free abelian group $N$ and a finite group $Q$ where the action of
$Q$ on $N$ factors through $\{\pm\text{Id}\}$, are \NP-complete \cite{ESW-virtAbel} (where decidability for  virtually cyclic and direct products  was established by \cite{FriedlLoh})
  \item  finitely generated abelian groups  is in \P 
       (via \cite{kannan1979}, see \cref{lem:ABEL-P}).
    \end{enumerate}
    Furthermore, \be\setcounter{enumi}{4}
     \item if $\cC$ is the class of non-abelian nilpotent groups, then $\Epi{\cC}{\cC}$ is undecidable \cite{Remeslennikov} 
       \item if $\cD$ is any class of groups which contains at least one finitely generated group, then $\Epi{\cD}{\Arb}$ is undecidable

       \item 
       $\Epi{\Arb}{\Arb}$ is recursively enumerable.

\ee
\end{propx}

Items (6) and (7) of the proposition
are shown in \Cref{subsec:Remaining}.

\subsection{Notation and basic facts}\label{subsec:Notation}

\subsubsection*{Words}
Let $X$ be a set. We call a finite sequence 
$(x_1,\dots, x_n)$ with $x_i\in X$ a \emph{word} over $X$, and denote it as $x_1\cdots x_n$. 
The set of all words over $X$ is denoted $X^\ast$.
The notation $u(X) = u(x_1,\dots, x_n)$ stands for a word $u$ over the letters $x_1,\dots, x_n$. 
If $Y = \{y_1,\dots,y_p\}, Z=\{z_1,\dots, z_q\}$ are sets, then we may write $u(Y,z_1,\dots,z_q)=u(y_1,\dots,y_p,z_1,\dots,z_q)=u(Y,Z)$.

For any set $X$ we let 
 $X^{-1}=\{x^{-1}\mid x\in X\}$ be a set of letters with $X\cap X^{-1}=\emptyset$.
If $X$ is a set, $H$ is a group
and $\psi\colon X\to H$ is a set map, we define the {induced} monoid homomorphism
from $(X\cup X^{-1})^\ast$ to $H$ by \[\psi(x_{1}^{\epsilon_1}\cdots x_{s}^{\epsilon_s})=\psi(x_{1})^{\epsilon_1}\cdots \psi(x_{s})^{\epsilon_s} \]
where $x_i\in X$ and $\epsilon_i=\pm 1$. 

\subsubsection*{Groups}
We use the notation $1_G$ or $1$ to denote the identity element of a group $G$, $[a,b]:= aba^{-1}b^{-1}$ for the {commutator} of two elements $a,b$ of $G$, and $[G,G]$ the commutator subgroup of $G$ (the subgroup consisting of all products of commutators of elements of $G$). For a group $G$ and two elements $a,b \in G$, ${}^ba = bab^{-1}$
denotes the {conjugation} of $a$ by $b$. If $u,v$ are two different ways to represent the same element of $G$, we write $u=_G v$.

We use $C_\infty$ to denote  the infinite cyclic group, $C_n$ for the cyclic group of order $n\in \N_+$, $A_n$ for  the alternating group on $n$ letters, and $D_{2n}$ for  the dihedral group of \emph{order} $2n$.
Recall that the \emph{center} of a group $G$, denoted $Z(G)$, is the set of $g\in G$ which commute with all elements in $G$.

\subsubsection*{Presentations}
A group $G\in \Arb$ is given by a finite presentation $\Gen{X}{R}$ where $X$ is a finite set and each $r\in R$ is a word over $X\cup X^{-1}$. We do not assume $X$ is a subset of $G$, so for example we may have $x,y\in X$ with $x=_Gy$, and we do not assume $G$ has decidable word problem, so, \emph{a priori}, given a presentation $\Gen{X}{R}$, we do not assume the existence of an  algorithm to determine whether $x=_Gy$ for $x,y\in X$.

Recall the following standard results. 

\begin{lemma}[von Dyck's lemma
]\label{lem:vonD}
    If $G$ is presented by $\displaystyle
    \Gen{g_1, \dots, g_n}{r_1, \dots, r_m}$
    where $r_i=r_i(g_1,\dots, g_n)$, and $\psi\colon \{g_1,\dots, g_n\}\to H$ is a set map to a group $H$, then the induced monoid homomorphism $\psi\colon \{g_1^{\pm 1}, \dots, g_n^{\pm 1}\}^\ast\to H$ defines a homomorphism from $G$ to $H$ if and only if $r_i(\psi(g_1),\dots,\psi(g_n))=_{H}1_H$  for $1\leq i\leq m$.
\end{lemma}

\begin{lemma}[Universal property of quotients]\label{lem:univQuotient}
 Let $N$ be a normal subgroup of a group $G$, with $\pi\colon G\to G/N$  the natural quotient map.  
If $\psi\colon G\to K$  is a homomorphism and $N\subseteq \ker\psi$ 
then there exists a unique homomorphism
$\widetilde\psi\colon G/N\to K$ such that $\widetilde\psi\circ \pi=\psi$.
Moreover, if $\psi$ is surjective, then so is $\widetilde\psi$.

\end{lemma}

\subsubsection*{Complexity}
We assume the reader is familiar with the complexity classes of $\P$ and $\NP$.
A decision problem $A$ is \emph{polynomial-time reducible} to a decision problem $B$ if there exists a polynomial-time computable function $f$ such that, for every instance $x$, the answer to $A$ on input $x$ is ``yes'' if and only if the answer to $B$ on input $f(x)$ is ``yes''. 
Moreover, a problem is \emph{\NP-hard} if every problem in \NP is reducible to it in polynomial time, and a decision problem is \emph{\NP-complete} if it is both in \NP and \NP-hard. 

When a finite presentation $\Gen{X}{R}$ is given as input, we assume the sets $X,R\subseteq (X\cup X^{-1})^*$ are given 
by listing letters and words in unary, so that the size of the input is $|X|+\sum_{r\in R}|r|$. Integers are assumed to be given in binary (for example, if the input is a free group of rank $d\in \N_+$).

\subsubsection*{Equations}
Let $\bbX = \{X_1,X_1^{-1},\dots,X_n,X_n^{-1}\}$.
An \emph{equation} over a group $G$ is a word
\begin{equation*}
    u(g_1,\dots,g_s,\bbX) 
\end{equation*}
where $g_i\in G$ for $i\in[1,s]$ are called \emph{constants} and $\bbX$ are called
\emph{variables}.
A \emph{system of equations} $\systemU$ is a finite list of equations $
  u_i(g_1,\dots,g_s,\bbX)$ for $i\in[1,m]$.
A \emph{solution} to a system of equations $\systemU$
is a map $\sigma\colon \bbX \to G$ of the form $\sigma\colon X_i \mapsto h_i, X_i^{-1} \mapsto h_i^{-1}$ for some $h_i \in G$, $i\in[1,n]$ so that 
\[u_i(g_1,\dots,g_s,\sigma(X_1),\sigma(X_1^{-1}), \dots, \sigma(X_n),\sigma(X_n^{-1}))=_G 1_G \ \ \textnormal{for all} \ \ i\in[1,m].\] 
A system of equations \emph{without constants} is a list of equations of the form   $u_i(\bbX)$ for $i\in[1,m]$.
Note that 
if $G$ is a finitely generated group with finite inverse-closed generating set 
$ \{y_1,\dots,y_s\}$, we may write any equation over $G$ as $ u(y_1,\dots,y_s,\bbX)$. 

For a group $H$ let $\Sol{H}$ denote the problem of deciding whether a system of equations over $H$ has a solution.
Goldmann and Russell proved the following.
\begin{theorem}[{\cite[Theorem 3]{GOLDMANN2002253}}] \label{thm:eqn_d2n_NPcomp}
    Let $H$ be a finite group. The problem  $\Sol{H}$ 
    is
    \be
        \item \NP-complete if $H$ is non-abelian
        \item in \P if $H$ is abelian.
    \ee
\end{theorem}

A key step 
to proving \cref{thm:MainDihedral}
is to reduce the problem of deciding whether a system of  equations over $D_{2n}$ with $n>1$ odd has a solution (which is \NP-hard by the above result) to epimorphism from $\Arb$ to $D_{2n}$, by constructing a finitely presented group from which an epimorphism to $D_{2n}$ exists if and only if the system has a solution.

\section{Preliminary results}

\subsection{Finite targets}

Let $\Fin$ denote the class of all finite groups.
We start by observing that $\Epi{\Arb}{\Fin}$ is  in \NP. 
We note that Holt and Plesken \cite[Chapter 7]{HoltP1989perfect} considered the computational problem of finding epimorphisms onto various classes of finite groups, without explicitly giving a complexity bound, and Friedl and L\"oh \cite[Proposition 5.2]{FriedlLoh} showed that $\Epi{\Arb}{\Fin}$ is decidable.

We assume that the input to the problem is a finite presentation $\Gen{g_1,\dots,g_n}{r_1,\dots,r_m}$ for the domain group $G\in \Arb$ and a multiplication table for the target group $Q\in \Fin$.

\begin{lemma}  \label{lem:epi_finite} 
    The problem    $\Epi{\Arb}{\Fin}$ is in \NP.
\end{lemma}

\begin{proof}
    On input a presentation $\Gen{g_1, \dots, g_n}{r_1,\dots, r_m}$ for  $G \in \Arb$, 
    non-deterministically specify 
    values $\tau(g_i)\in Q$ for each $i\in[1,n]$.

    Verify that $\tau$ 
    defines a homomorphism from $G$ to $H$  using \cref{lem:vonD} by checking that each relation is sent to $1_Q$ using the multiplication table for $Q$. 
  
    To verify that $\tau$ is a surjection, we may proceed as follows.
    Fix  a copy of $Q$ and `mark' each $q_j \in Q$ which satisfies $\tau(g_i) = q_j$ for $i\in[1,n]$. While not all of $Q$ is marked, scan the multiplication table for $Q$ to find $q_i,q_j,q_k\in Q$ so that $q_iq_j = q_k$ with $q_i,q_j$  marked and $q_k$  not marked, then mark $q_k$.
    If $\tau$ is a surjection then each $q\in Q$ is the image of some product of generators so the process will terminate with all of $Q$ marked.
    
    Each of the above steps takes polynomial time in the size $n$, $\sum_{i=1}^m\abs{r_i}$ and the size of the multiplication table for $Q$ (which is $O(\abs{Q}^2)$).
\end{proof}

Combining \cref{lem:epi_finite} with \cite[Corollary 1.2]{Kuperberg} 
we immediately have that $\Epi{\Arb}{\Fin}$ is \NP-complete.

\subsection{Dihedral groups}
\label{sec:MainDihedral}

Fix
\begin{equation}\label{eq:presD2n}
    \Gen{s, t}{s^2 = t^n = 1, \, sts = t^{-1}}
\end{equation}
as a presentation for $D_{2n}$. Using these relations, each element of $D_{2n}$ can be expressed uniquely as a word of the form $\alpha t^r$, where $\alpha \in \{1, s\}$ and $r \in [0, n-1]$.

The automorphism group of $D_{2n}$ is straightforward to compute (see, for example, \cite{Jerrythesis}).

\begin{lemma}\label{lem:Aut_D2n} For $r,\mu\in\Z$ let  $\varphi_{r, \mu}\colon D_{2n}\to D_{2n}$ be the map   $s\mapsto  st^r,  t\mapsto  t^\mu$.
    If $n \geq 3$, then
    $\Aut(D_{2n}) = \{\varphi_{r, \mu} \mid r \in [0, n-1], \mu \in [1, n-1], \gcd(\mu, n) = 1\}$.
  
\end{lemma}

\section{Odd Case}

Throughout this section we restrict our consideration to a dihedral group of order $2n$ with $n>1$ an odd integer.
Recall that for $\Epi{\Arb}{D_{2n}}$, the parameter $n$ is not part of the input, instead the input consists solely of a finite presentation for the domain group.

\begin{lemma}\label{lem:UsefulConjDn}
    For $n > 1$ is odd and $s,t$ generators in the presentation  \cref{eq:presD2n} above,
    $
    D_{2n} = \left\{\alpha_0 \left({}^{\alpha_1}t\right) \cdots \left({}^{\alpha_n}t\right) \mid \alpha_i \in \{1, s\}\right\}.
    $
\end{lemma}

\begin{proof}
    Since ${}^{\alpha_i}t \in \{t, t^{-1}\}$ for $\alpha_i \in \{1, s\}$, 
    \[
    \alpha_0 \left({}^{\alpha_1}t\right) \cdots \left({}^{\alpha_n}t\right) = \alpha_0 t^{n - \ell}t^{ - \ell} = \alpha_0 t^{n - 2\ell}
    \]
    where $\ell = \abs{\{i \in [1, n] \mid \alpha_i = s\}} \in [0, n]$. Using the 
 values in Table~\ref{table:odd_table} 
        and the fact that $t^n=1$,  \[\{\alpha_0 t^{n - 2\ell} \mid \alpha_0 \in \{1, s\}, \ell \in [0, n]\} = \{\alpha_0 t^r \mid \alpha_0 \in \{1, s\},  r \in [0, n-1]\}\]
    which proves the claim.
\end{proof}

 \begin{table}[h]\label{table:odd_table}
    \centering
   \[ \begin{array}{|c|c|c|c|c|c|c|c|c|c|}
        \hline
        \ell & 0 & 1 & \cdots & \tfrac{n-1}{2} & \tfrac{n+1}{2}   & \cdots  & n \\
        \hline
        n-2\ell & n & n -2  & \cdots & 1  & -1 & \cdots  & -n\\
              \hline
        (n-2\ell)+n & 0 &   & \cdots &  &   n-1 & \cdots & 0\\
        \hline
    \end{array}\]
    \caption{Computing exponents of $t$ in \cref{lem:UsefulConjDn}}
    \end{table}

\begin{lemma}\label{lem:ODDD2nCommute}
    If $n > 1$ is odd, $r \in [0, n-1]$, and $x \in D_{2n}$ commutes with $st^r$, then $x \in \{1, st^r\}$. In particular, the center $Z(D_{2n})$ is trivial. 
\end{lemma}

\begin{proof}
Write $x=\alpha t^\ell$ for $\alpha\in\{1,s\}$ and  $\ell\in[0,n-1]$. 
If $x=t^\ell$ then $[x,st^r]=t^{2\ell}=1$ if and only if $\ell=0$, so $x=1$. If $x=st^\ell$ then $[x,st^r]=t^{2r-2\ell}=1$ if and only if $n$ divides $2(r-\ell)$, and since $r,\ell\in[0,n-1]$ it follows that $r=\ell$. Thus the center is trivial for $n\geq 3$ since $x\in Z(D_{2n})$ implies $x$ commutes with $st$ and $st^2$.
 \end{proof}

\begin{definition}[Normal form]\label{defn:D2nNF_pt1}
    Let $n>1$ be an odd integer, $\bbX=\{X_1, X_1^{-1},\dots,X_k, X_k^{-1}\}$ and $\bbY=\{Y_{0,1}, Y_{0,1}^{-1},\dots,Y_{n,k}, Y_{n,k}^{-1}\}$. Define $\ONF \colon (\bbX\cup\{s,t,t^{-1}\})^\ast \to (\bbY\cup\{s,t,t^{-1}\})^\ast$ to be the monoid homomorphism induced by the set map 
 \[
        \begin{array}{lllllll}
            X_j \mapsto Y_{0,j}\cdot ({}^{Y_{1,j}}t)\cdot ({}^{Y_{2,j}}t) \cdots  ({}^{Y_{n,j}}t), &\quad &  X_j^{-1} \mapsto 
          ({}^{Y_{n,j}}t)^{-1}\cdots  ({}^{Y_{1,j}}t)^{-1} \cdot Y_{0,j}^{-1}; &\quad j \in [1,k] \\
            s \mapsto s, \\
            t \mapsto t, && t^{-1}\mapsto t^{-1}.
        \end{array}
        \]
\end{definition}
Recall that ${}^{Y_{i,j}}t=Y_{i,j}tY_{i,j}^{-1}$.

\begin{lemma}\label{lem:d2_p1_Sol_iff_D2nNF}
    Let $n,\bbX, \bbY$ and $\ONF$ be as \cref{defn:D2nNF_pt1}. 
    Let $\systemU$ be a system of equations in $D_{2n}$, where each equation $u_i \in (\bbX \cup \{s, t, t^{-1}\})^\ast$. There exists a solution $\sigma_1 \colon \bbX \to D_{2n}$ to $\systemU$ if and only if there exists a solution $\sigma_2 \colon \bbY \to \{1, s\} \subseteq D_{2n}$ to $(\ONF(u_i))_{m}$.
\end{lemma}

\begin{proof}
    For $i\in[1,m]$ each equation is a word $u_i(s,t,t^{-1}, X_1, X_1^{-1},\dots,X_k, X_k^{-1})$.
    By \cref{lem:UsefulConjDn}, replacing each variable $X_j$ by the word $Y_{0,j}\left({}^{Y_{1,j}}t\right) \cdots \left({}^{Y_{n,j}}t\right)=\ONF(X_j)$ in each equation and restricting $Y_{i,j}$ to take values in $\{1,s\}$ does not change the set of solutions. 
    Thus, we can rewrite each $u_i$ as $\ONF(u_i)$, and the result follows.
\end{proof}

Next,  build a finitely presented group from a system of equations over $D_{2n}$.

\begin{definition}
[Constructing a group presentation]
\label{defn:D2nNF_group}
    Let $n, \bbX,\bbY,\ONF$, 
   and $\systemU$ be as in the previous lemma, and 
   let $\cG_{n,k} = \{g_{i,j} \mid i\in[0,n],j\in[1,k]\}$  be a set of  $(n+1)k$ distinct letters.
    Define $\lambda\colon (\{s,t,t^{-1}\} \cup\bbY)^\ast \to (\{a, d,d^{-1}\} \cup \cG_{n,k}\cup\cG_{n,k}^{-1})^\ast$ as the monoid homomorphism  induced by the bijection
   \[  \lambda\colon\left\{\begin{array}{llllllll}
            s &\mapsto a \\
            t &\mapsto  d,& \quad &  t^{-1}&\mapsto d^{-1} \\
            Y_{i,j} &\mapsto g_{i,j}, &\quad&
            Y_{i,j}^{-1} &\mapsto g_{i,j}^{-1}; & \quad i \in [0,n],j\in[1,k].
    \end{array}\right.\]
We let $\DOddGroup(n,\systemU)$ denote the group with presentation
                \[\Gen
             { \{a, d\}\cup  \cG_{n,k} }{ \{a^2,  d^n, adad, \lambda(\ONF(u_i)), [g,g'], [g,a], g^2\mid  i\in[1,m], g,g' \in \cG_{n,k}    \}   }. \]
\end{definition}

\begin{remark}\label{rmk:OddDnPolyTimeConstruct}It is clear that for $n$  a fixed constant, the finite presentation for $\DOddGroup(n,\systemU)$ can be constructed in quadratic time in the size $k+\sum_{i=1}^m|u_i|$ of the system of equations. 
\end{remark}

The idea of this construction is to force that the system has a solution if and only if any epimorphism which maps $\DOddGroup(n,\systemU)$ onto $D_{2n}$  sends (for some automorphism $\varphi$  of $D_{2n}$) 
$a\mapsto \varphi(s)$, $d\mapsto\varphi(t)$, and $g_{i,j}$ to either  $1$ or $\varphi(s)$.
This is the content of the next two lemmas.

\begin{lemma}\label{lem:d2_pt1_restricted_epi}
    If  $\psi\colon \DOddGroup(n,\systemU) \to D_{2n}$ is an  epimorphism, then there exists 
    $\varphi \in \Aut(D_{2n})$ such that
    \begin{align*}
        \psi\colon \begin{cases}
            a &\mapsto \varphi(s) \\
            d &\mapsto \varphi(t) \\
            g_{i,j} &\mapsto  \gamma_{i,j} \in \gen{\varphi(s)}; \quad i \in [0,n], j \in [1,k]. 
        \end{cases}
    \end{align*}
\end{lemma}

\begin{proof}
For readability we denote $\DOddGroup(n,\systemU)$ as $G$ in this proof. 

If $\psi(d)=1$ then,  since $[a,g]$ is a relation in $G$ for all $g\in \cG_{n,k}$, $\gen{a, g\mid g\in \cG_{n,k}}$ is abelian and so $\psi(G)=\gen{\psi(a), \psi(g)\mid g\in \cG_{n,k}}$ is abelian,  which would mean that  $\psi$ is not surjective onto $D_{2n}$. Thus $\psi(d)\neq 1$.
If
$\psi(d)^2 = 1$, then the order
of $\psi(d)$ must divide $2$ and $n$, hence the order is $1$, but since $\psi(d) \neq 1$ this is a contradiction.
Thus $\psi(d)^2\neq 1$.
Write $\psi(d)=\alpha t^{p}$   with $p\in[0,n-1]$ and $\alpha\in\{1,s\}$, but since $st^pst^p=1$ in $D_{2n}$, it must be that 
$\psi(d)=t^p$.

If $\psi(a)=1$ then since $adad$ is a relation in $G$, $1=\psi(adad)=\psi(d^2)$ which is a contradiction. So $\psi(a)\neq 1$ and its order divides $2$ since $a^2$ is a relation in $G$. It follows that $\psi(a)=st^{r}$ for some $ r \in [0,n-1]\}$.

Again, since $[g_{i,j},a]$ is a relation in $G$ for all $g_{i,j}\in \cG_{n,k}$, $\psi(g_{i,j})$ commutes with $\psi(a)$, so by  \cref{lem:ODDD2nCommute},
$\psi(g_{0,1}),\dots,\psi(g_{n,k})  \in \{1,st^r\}=\gen{\psi(a)}$.
It follows that $\psi(G)=\gen{\psi(a),\psi(d)}=\gen{st^r,t^p}$.

To show that $\gcd(p,n)=1$ we proceed as follows. Since $\gcd(1,n)=1$,  $\varphi_{n-r,1}\circ\psi(G)=\gen{s,t^p}$ where  $\varphi_{n-r,1}\in\Aut(D_{2n})$.
Since $\varphi_{n-r,1}\circ \psi$ is surjective onto $D_{2n}$, we can write  $t\in D_{2n}$ as the image of some word over the generators of $G$, which means \[
t =(t^p)^{i_0}s(t^p)^{i_1}s\cdots s(t^p)^{i_{\ell-1}} s(t^p)^{i_\ell} \quad i_j\in \Z,\ell\in\N\]
which holds only if the number of $s$ letters is even, so $\ell=2m$, and then
\[t=
t^{i_0p} t^{-i_1p} \cdots t^{-i_{2m-1}p}t^{i_{2m}p}=(t^p)^{i_0-i_1+\dots+i_{2m}}
\] so $t^{1-px}=1$ where 
$ x=\sum_{j=0}^{2m} (-1)^ji_j \in \Z$, so $n$ divides $1-p x$ and so 
 $\gcd(p,n)=1$. 
 
Therefore, by \cref{lem:Aut_D2n}, there exists $\varphi_{r,p} \in \Aut(D_{2n})$ such that $\varphi_{r,p}(t) = t^{p}=\psi(d)$ and  $\varphi_{r,p}(s) = st^r=\psi(a)$.
\end{proof}

\begin{lemma}\label{lem:d2np1_iff_epi_and_sol}
    Let $n>1$, $\bbX = \{X_1,X_1^{-1}, \dots, X_k,X_k^{-1}\}$ and $\systemU$ with   $u_i \in (\{s,t,t^{-1}\}\cup \bbX)^\ast$ be a system of equations over $D_{2n}$. There exists an epimorphism $\psi\colon \DOddGroup(n,\systemU) \to D_{2n}$ if and only if there exists a solution $\sigma\colon \bbX \to D_{2n}$ to the system  $\systemU$.
\end{lemma}

\begin{proof}
   Again let $G=\DOddGroup(n,\systemU)$. Assume that there exists an epimorphism $\psi'\colon G\to D_{2n}$.
    By \cref{lem:d2_pt1_restricted_epi} there exists $\varphi \in \Aut(D_{2n})$ such that 
    \begin{align*}
        \psi'&\colon \begin{cases}
            a &\mapsto \varphi(s) \\
            d &\mapsto \varphi(t) \\
            g_{i,j} &\mapsto  \gamma_{i,j}' \in \gen{\varphi(s)}; \quad i \in [0,n], j \in [1,k] 
        \end{cases}
    \intertext{so letting $\psi = {\varphi^{-1}} \circ \psi'$ yields  an epimorphism}
        \psi&\colon \begin{cases}
            a &\mapsto s \\
            d &\mapsto t \\
            g_{i,j} &\mapsto  \gamma_{i,j} \in \gen{s}; \quad  i \in [0,n], j \in [1,k]. 
        \end{cases}
    \end{align*}

    Define $\sigma\colon \bbY \to \{1,s\}$ by $\sigma(Y_{i,j}) = \gamma_{i,j}, \sigma(Y_{i,j}^{-1}) = \gamma_{i,j}^{-1}$.
    Note that since $s^2 = 1$, then for all $Y_{i,j} \in \bbY$, $\sigma(Y_{i,j}) = \sigma(Y_{i,j}^{-1})$, so without loss of generality we may assume $\bbY=\{Y_{0,1}, \dots, Y_{n,k}\}$. For $i\in[1,m]$ let $v_i\in(\{s,t,t^{-1}\}\cup\bbY)^*$ be such that $\ONF(u_i)=v_i$. 
    Since $\psi$ is a homomorphism, for each 
    $i \in [1,m]$,
     \begin{align*}
        1 = \psi(\lambda(\ONF(u_i))) &= \psi(\lambda(v_i(s,t,t^{-1}, Y_{0,1},\dots,Y_{n,k}))) \\
        &= \psi(v_i(a,d,d^{-1}, g_{0,1}, \dots, g_{n,k})) \\
        &= v_i(s,t,t^{-1}, 
        \gamma_{0,1}, \dots, \gamma_{n,k})\\
        &= \sigma(v_i(s,t,t^{-1},Y_{0,1},\dots,Y_{n,k})) = \sigma(\ONF(u_i))
    \end{align*}
    so $\sigma$ solves $(\ONF(u_i))_{[1,m]}$, and the result follows by \cref{lem:d2_p1_Sol_iff_D2nNF}.

Conversely, 
assume there exists a solution 
to $\systemU$, so by \cref{lem:d2_p1_Sol_iff_D2nNF} there exists a solution $\sigma\colon \bbY \to \{1,s\}$ to $(\ONF(u_i))_{[1,m]}$, so for $i \in [1,m]$, if $\ONF(u_i) = v_i$  then
    \begin{equation}\label{eqn:d2_pt1_ifSol_Hom}
        \sigma(\ONF(u_i)) = \sigma(v_i(s,t,t^{-1}, \bbY))=1. 
    \end{equation}

    Define $\psi\colon \{a,d,d^{-1}\}\cup\cG_{n,k}\cup\cG_{n,k}^{-1} \to D_{2n}$ as the set map
 \[
        \psi\colon \left\{\begin{array}{llllll}
            a &\mapsto s \\
            d &\mapsto t, &\quad &d^{-1}\mapsto t^{-1}\\
            g_{i,j} &\mapsto \sigma(Y_{i,j}) &
            \quad & g_{i,j}^{-1}  \mapsto \sigma(Y_{i,j})^{-1}; &
            \quad  &  g_{i,j} \in \cG.
        \end{array}\right.
\]

    Since the other relations in $G$ clearly map to $1$ in $D_{2n}$, 
    by \cref{lem:vonD},  $\psi$ induces a homomorphism from $G$ to $D_{2n}$ if and only if $\psi(\lambda(\ONF(u_i))) = 1$ for all $i\in[1,m]$.

It follows that 
    \begin{align*}
        \psi(\lambda(\ONF(u_i))) &= \psi(\lambda(v_i(s,t,t^{-1}, \bbY)))\\ 
        &= v_i(s,t, t^{-1},\sigma(\bbY))= \sigma(v_i(s,t, t^{-1},\bbY))=1
        \quad \textnormal{ by \cref{eqn:d2_pt1_ifSol_Hom}}
    \end{align*}
   so $\psi$ is a homomorphism, which is surjective since  $\psi(G) = \gen{s,t} = D_{2n}$.
\end{proof}

We now have the main result for this section.
\begin{theorem}\label{thm:mainODD} Let $n>1$ be an odd integer. The problem 
$\Sol{D_{2n}}$
     is 
    polynomial time reducible to $\Epi{\Arb}{D_{2n}}$ (and therefore $\Epi{\Arb}{D_{2n}}$ is \NP-hard).
\end{theorem}

\begin{proof}
    Given an input system of equations $\systemU$ with variables $\bbX = \{X_1, X_1^{-1}, \dots, X_k, X_k^{-1}\}$ over $D_{2n}$, we can construct the group $G=\DOddGroup(n, \systemU)$  in polynomial time by \cref{rmk:OddDnPolyTimeConstruct}. By \cref{lem:d2np1_iff_epi_and_sol}, a solution to $\systemU$ exists if and only if there exists an epimorphism from $G$ to $D_{2n}$, and so deciding whether or not there is an epimorphism from  $G$ to $D_{2n}$ is polynomial time reducible to 
    deciding whether or not the system $\systemU$ over $D_{2n}$ has a solution. 

    It follows that   $\Epi{\Arb}{D_{2n}}$ is \NP-hard by \cref{thm:eqn_d2n_NPcomp}.
\end{proof}

\section{Reduction to Quotient by the center}

We now prove a transfer result 
which gives \cref{prop:CenterQuotient}.

\begin{lemma}\label{lem:QuotientCenter}
Let $H\in \Arb$ with  $Z(H)$ finitely generated. The problem 
	 $\Epi{\Arb}{H/Z(H)}$ is  polynomial time reducible to $\Epi{\Arb}{H} $. 
     
     In particular, if  
      $\Epi{\Arb}{H/Z(H)}$ is \NP-hard then $\Epi{\Arb}{H} $ is \NP-hard. 
  \end{lemma}

\begin{proof}
Let $\Delta$ denote a finite generating set of $Z(H)$.
Given $G = \Gen{\Sigma}{R}$ 
as an instance of $\Epi{\Arb}{H/Z(H)}$, 
 construct a presentation  $\widetilde G = 
 \langle\widetilde \Sigma \mid \widetilde R\rangle$ with
\smallskip
\begin{align*}
    \widetilde \Sigma &= \Sigma \cup X \cup Y \qquad \quad\text {with } X = \{\,x_r\mid r\in R\,\}, Y = \{\,y_c \mid c \in \Delta\,\},\\
	\widetilde R &= \{\, r\,  x_r ^{-1} \mid r \in R\,\}\: \cup\: \{\, [a,x] \mid a\in \widetilde \Sigma,\, x \in X \cup Y\,\}.
\end{align*}

Clearly, this presentation is finite and can be computed in polynomial time. Therefore, it remains to show that $G \twoheadrightarrow H/Z(H)$ if and only if $ \widetilde G \twoheadrightarrow H$.

Let $\psi\colon \widetilde G \twoheadrightarrow H$ be an epimorphism, and let
 $\pi_H\colon H\to H/Z(H)$ and $\pi_G\colon \widetilde G\to \widetilde G/\gen{X\cup Y}$ be the respective quotient maps.
Since $\gen{X \cup Y} \leq Z(\widetilde G)$,  
 $\psi(\gen{X \cup Y}) \subseteq Z(H)=\ker(\pi_H)$, so $\gen{X\cup Y}\subseteq \ker( \pi_H\circ \psi)$.
This yields an epimorphism \[\pi_H\circ \psi\colon \widetilde G\to H/Z(H)\] and since  $\gen{X\cup Y}\subseteq \ker( \pi_H\circ \psi)$, by \cref{lem:univQuotient} there exists a unique epimorphism $\tilde \psi$ as in \cref{fig:commutativeDiag}.

 \begin{figure}[h]
\begin{center}
\begin{tikzcd}[column sep=3em, row sep=2.5em, ampersand replacement=\&]
    \widetilde{G}
        \arrow[r, two heads, "\psi"]
        \arrow[dr, two heads, "\pi_G"']
    \& H
        \arrow[r, two heads, "\pi_H"]
    \& H/Z(H) \\
    \&
    G \cong \widetilde{G}/\gen{X\cup Y}
        \arrow[ur, dashed, two heads, "\tilde{\psi}"']
    \&
\end{tikzcd}  
\end{center}\caption{Commutative diagram in the proof of  
\cref{lem:QuotientCenter}.
\label{fig:commutativeDiag}}
\end{figure}

For the other direction, let $\phi: G \twoheadrightarrow H/Z(H)$ be an epimorphism, and let $F(\widetilde\Sigma)$ denote the free group with free basis $\widetilde\Sigma$.
Define a homomorphism $\psi: F(\widetilde\Sigma) \to H$ via the set map 	
\begin{itemize}
	\item for $a\in \Sigma$, let $\psi(a)=h$ for some choice of $h$ with $\phi(a)=hZ(H)$  
	\item	for $x_r\in X$ where $r=a_1\cdots a_n\in \Sigma^*$, set $\psi(x_r) = \psi(a_1)\cdots \psi(a_n)$ (so $\psi(x_r)=\psi(r)$)
	\item	for $y_c\in Y$, set $\psi(y_c) = c\in\Delta$.
\end{itemize}

We first show that $\psi$ is surjective.
Let $h\in H$. Since $\phi$ is surjective, the coset $hZ(H)$ has preimage  some $g \in G=\gen{\Sigma} \leq F(\widetilde\Sigma)$.  Let $g=a_1\cdots a_s$ with $a_i\in \Sigma$. Then $\phi(a_i)=\psi(a_i)Z(H)$ by definition of $\psi$, so $\psi(g)Z(H)=\phi(g)=hZ(H)$.
  Then $h=\psi(g)c$ for $c\in Z(G)=\gen{\Delta}$. Writing $c=c_1\cdots c_t$ with $c_i\in\Delta$,  $k=y_{c_1}\cdots y_{c_t}\in \gen{Y}\leq F(\widetilde\Sigma)$ satisfies $\psi(k)=c$.
  Thus $h=\psi(gk)$ proving that $\psi$ is surjective.

Now, it  remains to show that $\psi$  induces a homomorphism $\widetilde G \to H$. 
For each $r\in R$,  $\psi(r  x_r^{-1} )=\psi(r)\psi(r)^{-1}=1$ by definition of $\psi$.
For each $x_r\in X$,  $\phi(r)=1$ so $\psi(x_r)=\psi(r)\in Z(H)$, so $[a,x_r]\in \ker(\psi)$ for any $a\in \widetilde \Sigma$, and clearly $[a,y]\in \ker(\psi)$ for each $y_c\in Y$ since $\psi(y_c)=c\in \Delta\subseteq Z(H)$.
Thus $\langle\widetilde R\rangle\subseteq \ker(\psi)$, so by 
\cref{lem:univQuotient} there is a 
homomorphism from $\widetilde G\cong F(\widetilde \Sigma)/\langle\widetilde R\rangle$ to $H$. 
\end{proof}

\begin{corollary}
For $n$ even and not a power of $2$, the problem 
    $\Epi{\Arb}{D_{2n}}$ in \NP-hard.
\end{corollary}
\begin{proof}

Let $r\in[0,n-1]$.
Then $[t,st^r]=tst^rt^{-1}t^{-r}s=t^{2}\neq 1$ since $n>2$, so $st^r$ is not in the center $Z(D_{2n})$.
Also $[t^r,s]=t^{2r}=1$ if and only if  $r=0$ or $\frac{n}2$.
Therefore 
$Z(D_{2n})=\{1,t^{\nicefrac{n}2}\}$, so $D_{2n}/Z(D_{2n})\cong D_n$ and the  result then follows  from 
\cref{lem:QuotientCenter,thm:mainODD} by induction.
\end{proof}

This completes the proof of  \cref{thm:MainDihedral}.
    As noted above, $D_{2n}$ is nilpotent if and only if $n$ is a power of $2$. 
    It is unclear whether \cref{thm:MainDihedral} extends to these, with the smallest example being $D_8$. 
   \begin{open}\label{qn:Deight}
  Is  $\Epi{\Arb}{D_8}$ \NP-hard, in \P, or otherwise? 
 \end{open}

In \cite{ESW-virtAbel} the authors prove that epimorphism for $\mathcal T$ the class of direct products of the form $A\times B$ with $A$ finitely presented  abelian and $B$ finite, $\Epi{\Arb}{\cT}$
is in \NP. It follows that  $\Epi{\Arb}{A\times B}$ 
is in \NP. 
Combining this with \cref{lem:QuotientCenter} we obtain the following.

\begin{corollary}[Direct product of trivial center with abelian]\label{lem:AbelianTimesNoCenter}
    Let  $A$ be a finitely generated abelian group  and $B$ a finite group with trivial center. If   $\Epi{\Arb}{B}$ is \NP-hard
   then $\Epi{\Arb}{B\times A}$ is \NP-complete.
\end{corollary}
\begin{proof}
    Since $Z(B)$ is trivial, $(A\times B)/Z(A\times B)$  is isomorphic to $B$, so  $\Epi{\Arb}{A\times B}$ is \NP-hard by \cref{lem:QuotientCenter},  and is in \NP by \cite{ESW-virtAbel}.
\end{proof}

 So, for example,  epimorphism onto $A\times A_m$ or $A\times D_{2n}$ for  $A$ finitely generated abelian, $m\geq 5$ and $n>1$ odd, are  \NP-complete.

\section{Brief survey of known results}
\label{sec:otherResults}

This section 
 collects together some known facts and  observations to add to the list of cases for which bounds
 on the complexity of the epimorphism problem can be given.

\subsection{Epimorphism onto free groups}
\label{subsec:FreeTargets}

Recall the notation for systems of equations from Subsection~\ref{subsec:Notation}.
Let $m,n,d\in \N_+$, $\bbX = \{X_1,X_1^{-1},\dots,X_n,X_n^{-1}\}$, $F_d$ be a free group of rank $d$ with identity element $1\in F_d$,  and 
  $ \left( u_i\right)_{[1,m]}$
a system of equations without constants over $F_d$, so   
 each $u_i=u_i(\bbX)\in \bbX^*$.
Define the \emph{rank} of a solution $\sigma\colon \bbX\to F_d$ to 
 be the rank of the free subgroup $\langle \sigma(X_1),\dots, \sigma(X_n)\rangle$ of $F_d$. Since $\{X_i\mapsto 1\}$ is a solution to any system without constants, the rank of a solution $\sigma$ is well defined. Moreover it is at least $0$  and at most the number of variables $n$.
Define the \emph{rank} of the system without constants $(u_i)_{ i\in[1,m]}$  to be the maximum rank over all solutions.

As an example, suppose $u=X^{-1}Y^{-1}XYZ^{s}$ is a system of one equation over $F_d$. 
By \cite{Schutzenberger} (see \cite[page 51]{LS}) 
the only solutions are $\{X\mapsto v^i, Y\mapsto v^j, Z\mapsto 1\}$ for some $v\in F_d$, so 
the rank of this equation is $1$.

\begin{theorem}[{Razborov \cite[Theorem 3]{Razborov1985}}]\label{thm:Raz}
    Let $d$ be a fixed integer.  
    Given a system of equations $ \left( u_i\right)_{[1,m]}$ without constants over a free group $F_d$, there is an algorithm which 
    computes the {rank} of the system. 
\end{theorem}

Note that Razborov's algorithm runs via constructing ``Makanin-Razborov diagrams'', and the complexity to compute these are not known, most likely at least doubly exponential space (see for example \cite[page 2]{VolkerLothaire}).

\begin{lemma}\label{lem:epiFree}
    Let $\cT$ be a class consisting of finite rank free groups. The problem 
    $\Epi{\Arb}{\cT}$ is decidable.
\end{lemma}

    \begin{proof}	
    Assume the input is a finite presentation $\Gen{g_1,\dots,g_n}{r_1,\dots,r_m}$ for a group $G \in \Arb$, and $d \in \N_+$ specifies a target free group of rank $d$.
    Let $\cG= \{g_1,g_1^{-1},\dots,g_n,g_n^{-1}\}$ and $\bbX=\{X_1,X_1^{-1},\dots,X_n,X_n^{-1}\}$.
    Perform the following procedure:
    \begin{enumerate}
    \item Fix a free group $H=\gen{a_1,\dots,a_d}$.
    \item Let $\lambda\colon g_i \mapsto X_i,  g_i^{-1} \mapsto X_i^{-1}$ induce a monoid homomorphism from words over $\cG$ to words over $\bbX$. 
    Then  $(\lambda(r_j))_{[1,m]}$
    is a system of equations without constants over $F_d$ with each $\lambda(r_j)\in\bbX^*$. 
    \item Apply \cref{thm:Raz} with input $(\lambda(r_j))_{[1,m]}$ over the group $H$, to compute the rank $\MFr\in[0,n]$ of the system.
    If $\MFr\geq d$, return `Yes', else return `No'.
    \end{enumerate}

    The justification for the procedure is as follows. Let $h_1,\dots, h_n \in H$. By \cref{lem:vonD}, ${\{X_i\mapsto h_i, X_i^{-1}\mapsto h_i^{-1}\mid i \in [1,n]\}}$ is a solution to $(\lambda(r_j))_{[1,m]}$
    if and only if the map induced by $\{g_i\mapsto h_i\mid i \in [1,n]\}$ is a homomorphism from $G$ to $H$. If $\MFr\geq d$ 
    then there is a homomorphism $\kappa$ from $G$  onto a subgroup $K$ of $H$ such that $K$ is free of rank $\MFr$. 
    Then $K$ has some free basis, say $\{y_1,\dots, y_{\MFr}\}$ (which we do not have to compute) and there exists a subgroup $K'=\gen{y_1,\dots,y_d}$ of $H$ and map   $\tau\colon K \to K'$ defined by $\tau(y_i)=\begin{cases}
       y_i & i\leq d\\
       1 & i>d
    \end{cases}$. Then $\tau\circ \kappa$ is an epimorphism from $G$ to $K'$, and by definition $K'$ is free of rank $d$. (Note that the procedure finds an epimorphism to some free group of rank $d$, not necessarily to the original group  $H$.)  
    If $\MFr<d$, then there is no epimorphism since an epimorphism $\psi\colon G\to H$ would be a solution to the system of rank $d$.   
\end{proof}

\begin{open}
    Determine bounds for the complexity of computing the rank of a system of equations $ \left( u_i\right)_{[1,m]}$ without constants over a free group $F_d$. For  example can this be done in \EXPSPACE or \PSPACE?
\end{open}

\subsection{Epimorphism onto non-abelian finite simple groups}\label{subsec:Kuperberg}

Let $\Hom$ be the set of all triangulated homology 3–spheres, given by finite triangulations. From a finite triangulation of a 3-manifold $M$ one can write down (in linear time) a finite presentation for the fundamental group $\pi_1(M)$ of the manifold. Let $H$ be a fixed, finite, non-abelian simple group. Kuperberg and Samperton \cite[Corollary 1.2]{Kuperberg} prove the following problem is \NP-complete.

\compproblem[]{$M\in\Hom$ and the promise that every non-trivial homomorphism from $ \pi_1(M)$ to $H$ is surjective}{Is there a non-trivial homomorphism from $ \pi_1(M)$ to $H$?}

It follows immediately that $\Epi{\Arb}{H}$ is \NP-hard, via the following reduction. On input a finite triangulation for $M$ and the promise that every non-trivial homomorphism from $ \pi_1(M)$ to $H$ is surjective, obtain in linear time a presentation for $\pi_1(M)$. Then there is an epimorphism from  $\pi_1(M)$ to $H$ if and only if there is a non-trivial homomorphism from  $\pi_1(M)$ to $H$.

This gives Item~(2) of \cref{prop:ObserveFacts}.

\subsection{Epimorphism onto abelian groups}\label{subsec:AbelianCase}

Let $\Abe$ denote the class of finitely generated abelian groups. In this subsection we show that  
$\Epi{\Arb}{\Abe}$ is in \P. 

\medskip\noindent
 \emph{Notation.} For a group $G$, $G_{ab} = G/[G:G]$ is the \emph{abelianisation} of $G$. Note that if $\Gen{S}{R}$ is a presentation 
for $G$ then $\Gen{S}{R\cup\{[s,t]:s,t\in S\}}$ is a presentation for $G_{ab}$, which can clearly be obtained in time quadratic in $|S|$. 
Recall that $C_a$ denotes the cyclic group of order $a\in \N_+\cup\{\infty\}$. 
By the well-known fundamental theorem of finitely generated abelian groups, 
if $G$ is a finitely generated abelian group then
\begin{equation*}
    G \cong (C_\infty)^{d} \times C_{a_1} \times \cdots \times C_{a_k}
\end{equation*}
for some $d, a_1,\dots,a_k \in \N_+$ where $a_i \mid a_{i+1}$ for $i\in[1,k-1]$.  Such a description for a finitely generated abelian group will be called its \emph{standard form}.

Let $G = \Gen{\cX}{\cR}$, then  $\kappa\colon x \mapsto x$ for all $x \in \cX$ induces a homomorphism $\kappa\colon G \to G_{ab}$. Moreover $\kappa$ is surjective since each $g\in G_{ab}$ can be represented by a  word  $w \in (\cX\cup\cX^{-1})^\ast$, and the element $g'\in G$ spelled by $w$ satisfies $\kappa(g')=g$. 
It follows that if $\tau\colon G_{ab} \to H$ is an epimorphism to a group $H$, then $\psi\colon G \to H$ defined by $\psi(g) = \tau(\kappa(g))$ for all $g \in G$ is an epimorphism. When $H$ is abelian, there is a stronger statement.

\begin{lemma}\label{lem:abe_epi_iff_abeabe_epi}
    Let $G \in \Arb$ and $H \in \Abe$. There exists an epimorphism $\psi\colon G \to H$ if and only if there exists an epimorphism $\varphi\colon G_{ab} \to H$. 
\end{lemma}

The following technical lemma that allows us to check if an epimorphism exists between two finitely generated abelian  groups given in standard form. 

\begin{lemma}\label{lem:epi_abe_tech_lem}
     Let 
    \begin{equation*}
        G \cong (C_\infty)^d \times C_{a_1} \times \cdots \times C_{a_s}
\  \textnormal{ and }\
        H \cong C_{c_1} \times \cdots \times C_{c_t}
    \end{equation*}
    with $s,t,a_i,c_j \in \N_{+}$ such that $c_{i+1} \mid c_{i}$ and $a_{i+1} \mid a_i$. There exists an epimorphism $\psi\colon G \to H$ if and only if $s \geq t-d$ and $c_{d+i} \mid a_i$ for $i \in [1,t-d]$.
\end{lemma}
See~\cite{Jerrythesis} for a proof.
\begin{theorem}[Computing Smith normal form; \cite{kannan1979}]
    \label{thm:snf_poly} The following computational problem is in \P. 

    \compalgo{a matrix $A\in 
\Z^{m\times n}$}{compute $K \in \GL(m,\Z), L \in \GL(n,\Z), D\in \Z^{m\times n}$ such that 
$A=KDL$, \begin{equation*}
        D = \begin{pmatrix}
            \begin{array}{c|c}
                M & 0 \\
                \hline
                0 & 0
            \end{array} 
            \end{pmatrix},
        M = \begin{pmatrix}
            \MFd_1 & 0 & \cdots & 0 \\
            0 & \MFd_2 & \cdots & 0 \\
            \vdots & \vdots & \ddots & \vdots \\
            0 & 0 & \cdots & \MFd_r
        \end{pmatrix}
    \end{equation*}
       with $0 \leq r \leq \min(m,n)$,  each $\MFd_i > 0$ 
        and $\MFd_i \mid \MFd_{i+1}$ for all $i \in [1,r-1]$. }
\end{theorem}

Using \cref{thm:snf_poly}, on input finite presentations for $G_1$ and $G_2$, we can compute integers  $s,t,a_i,c_j \in \N_{+}$ such that $c_{i+1} \mid c_{i}$, $a_{i+1} \mid a_i$ and 
 \begin{equation*}
        G_1 \cong  (C_\infty)^{d_1} \times C_{a_1} \times \cdots \times C_{a_s}
  \ \textnormal{ and }\
        G_2 \cong  (C_\infty)^{d_2} \times C_{c_1} \times \cdots \times C_{c_t}.
    \end{equation*}
     We require the free abelian rank of $G_1$ to be at least that of $G_2$ for an epimorphism to be possible. If so, applying \cref{lem:epi_abe_tech_lem} to \begin{equation*}
         (C_\infty)^{d_1-d_2} \times C_{a_1} \times \cdots \times C_{a_s}
  \ \textnormal{  and  }\
     C_{c_1} \times \cdots \times C_{c_t}
    \end{equation*}
    we can decide  $\Epi{\Abe}{\Abe}$  in \P.
    Putting this all together we obtain
\begin{lemma}\label{lem:ABEL-P} Let $\cT$ be a class of finitely generated abelian groups. Then 
    $\Epi{\Arb}{\cT}$ is in \P, where the input for both domain and target groups are finite presentations. 
\end{lemma}

\begin{proof}
    Let $G \in \Arb$ and $H \in \Abe$. The following procedure solves our problem.
    \begin{enumerate}
        \item Compute a presentation for 
        $G_{ab}$.
        \item If  $\Epi{\Abe}{\Abe}$ on input $G_{ab}$ as the domain group and $H$ as the target group returns `Yes, return `Yes'. Else return `No'.
    \end{enumerate}
    Step~(1) constructs a finite presentation for $G_{ab}$ in polynomial time, and  step~(2) is in \P as described above.
\end{proof}

\subsection{Remaining items in \cref{prop:ObserveFacts}}\label{subsec:Remaining}

To see that  $\Epi{\cD}{\Arb}$ is undecidable when $\cD$ is any class of groups which contains at least one finitely generated group, suppose for contradiction that 
$\Epi{\cD}{\Arb}$ is decidable. On input a finite presentation for a group $H$ presentation, the following algorithm decides whether or not $H$ is trivial (which is undecidable).
Let $G\in \cD$ be a group with $n$ generators, and $K$ the free product of $n+1$ copies of $H$.  Recall Grushko's theorem \cite{Gruschko,Neumann}
which states that $m(A\ast B)=m(A)+m(B)$ where $m(C)$ the  minimum number of generators required to generate a group $C$. Then  $m(K)=(n+1)m(H)$ and $m(G)\leq n$.
It follows that 
there is an epimorphism from $G$ to $K$ if and only if $H$ is trivial.

To see that 
$\Epi{\Arb}{\Arb}$ is recursively~enumerable, recall that finitely presented groups have recursively~enumerable word problem. On input finite presentations for $G,H\in\Arb$ with $G,H$ having generating sets $\cX=\{x_1,\dots, x_m\},\cY$
respectively, enumerate $m$-tuples $\{h_1,\dots, h_m\}$ of elements of $H$ and $n$-tuples of words $\{u_1,\dots, u_n\}\subseteq \cX^*$ and recursively verify \be\item the map $x_i\mapsto h_i$ is a homomorphism by (recursively) solving the word problem in $H$ and \cref{lem:vonD}
\item 
$u_i$ is a preimage of $y_i$ under this map by (recursively) solving the word problem in $H$. 
\ee

\section*{Acknowledgements}
This research was supported by Australian Research Council grant DP210100271. The second author 
was supported by an Australian Government Research Training Program Stipend.
The third author was partially supported by DFG grant WE 6835/1--2.
The first and third author were 
supported through the program \emph{Oberwolfach Research Fellows} at the Mathematisches Forschungsinstitut Oberwolfach in 2023.

The authors wish to thank Michal Ferov, Alex Levine,  Eric Samperton, Robert Tang, Alex Thumm and Kane Townsend 
for fruitful discussions about this project, and sincerely thank
the anonymous reviewer for their careful reading and detailed and helpful feedback which helped shape the present version of this work.

\bibliographystyle{plain}
\bibliography{refs}

\end{document}